\documentclass[11pt,reqno,a4paper]{amsart}

\usepackage{etoolbox}
\usepackage{amsmath}
\usepackage{enumerate}
\usepackage{amssymb}
\usepackage{amscd}
\usepackage{amsthm}
\usepackage{amsfonts}
\usepackage{graphicx}
\usepackage[matrix, arrow, curve]{xy}
\usepackage{comment}
\usepackage{xcolor}
\usepackage[scaled]{helvet} 
\usepackage{courier} 
\usepackage[T1]{fontenc}                                                    
\usepackage{hyperref}
\usepackage{footmisc}

\normalfont

\numberwithin{equation}{section}

\newcommand{\cD}{\mathcal{D}}

\newcommand{\cR}{\mathcal{R}}

\newcommand{\cU}{\mathcal{U}}

\newcommand{\N}{\mathbb{N}}

\newcommand{\R}{\mathbb{R}}

\newcommand{\Z}{\mathbb{Z}}

\newcommand{\oB}{\overline{B}}

\DeclareMathOperator{\dist}{dist}

\DeclareMathOperator{\diam}{diam}
\DeclareMathOperator{\mesh}{mesh}

\DeclareMathOperator{\asdim}{asdim}

\theoremstyle{plain}
\newtheorem{theorem}{Theorem}[section]

\newtheorem{lemma}[theorem]{Lemma}

\theoremstyle{definition}
\newtheorem{definition}[theorem]{Definition}

\newtheorem{remark}[theorem]{Remark}
\newtheorem{example}[theorem]{Example}

\patchcmd{\subsection}{-.5em}{.5em}{}{}
\patchcmd{\subsubsection}{-.5em}{.5em}{}{}

\begin{document}

\title{The (largest) Lebesgue number and its relative version}
\author{Vera Toni\' c}

\address{Department of Mathematics\\
University of Rijeka\\
51000 Rijeka\\ Croatia}
\email{vera.tonic@math.uniri.hr}

\begin{abstract}
In this paper we compare different definitions of the (largest) Lebesgue number of a cover $\cU$ for a metric space $X$. We also introduce the relative version for the Lebesgue number of a covering family $\cU$ for a subset $A\subseteq X$, and justify the relevance of introducing it by giving a corrected statement and proof of the Lemma 3.4 from \cite{BuyaloLebedeva}, involving $\lambda$-quasi homothetic maps with coefficient $R$ between metric spaces and the comparison of the mesh and the Lebesgue number of a covering family for a subset on both sides of the map.
\end{abstract}
\maketitle

\section{Introduction}

In any basic course on point-set topology, one is usualy introduced to the notion of \emph{a Lebesgue number} of an open cover for a (non-empty) compact metric space, in the form of Lebesgue covering lemma, which is commonly given in one of the two following versions:

\begin{lemma}[Lebesgue covering lemma, Version $1$] (\cite[Thm.IV.5.4]{Hu}, \cite{Kelley}, \cite{Dugunji}, \cite{Engelking})\label{Leb-Version1}
For every open cover $\cU$ of a compact metric space $X$, there exists a real number $\lambda>0$ such that each open ball of radius $\lambda$ has to be contained in some element of $\cU$. That is, the family of open balls $\{ B(x,\lambda)\ | \ x\in X\}$ is a refinement of $\cU$. Any such $\lambda$ is referred to as a Lebesgue number of the cover $\cU$.
\end{lemma}

\begin{lemma}[Lebesgue covering lemma, Version $2$](\cite[Cor.IV.5.6]{Hu}, \cite{Munkres}) \label{Leb-Version2}
For every open cover $\cU$ of a compact metric space $X$, there exists a real number $\overline\lambda>0$ such that any subset $A$ of $X$ with $\diam A<\overline\lambda$ has to be contained in some element of $\cU$. Any such $\overline\lambda$ is referred to as a Lebesgue number of the cover $\cU$.
\end{lemma}
\begin{remark}\label{remark: a Lebesgue leq}
Note that in Version $2$, the strict inequality  $\diam A<\overline\lambda$ is often replaced by  $\diam A\leq\overline\lambda$ (see, for example, \cite{Ungar}, \cite{Coornaert}).
\end{remark}

It is easy to see that the statements from Version $1$ and Version $2$ above are equivalent: any $\lambda$ satisfying Version $1$ will appear among the numbers $\overline\lambda$ satisfying the requirement of Version $2$, and for a $\overline\lambda$ from Version $2$, any $\lambda \in \left(0,\frac{\overline\lambda}{2}\right)$ will satisfy the Version $1$ requirement.

\medskip

But our intention here is to discuss what could be referred to as \emph{the largest Lebesgue number} of a cover $\cU$ for a metric space $X$, which is a notion often used in asymptotic dimension theory.
However, note that calling it ``the largest'' may be a little misleading, since ``the largest'' is often associated with the word ``maximal''. And we shall see that this number can be defined using supremum, as well as a combination of $\inf$ and $\sup$ of certain sets of values (see Definitions \ref{def: Lebesgue BS}, \ref{def: L-Rad} and \ref{def: L-Diam} below), and there are cases when supremum exists, while the maximum does not (see Example \ref{ex-balls-1} below). So perhaps it would be more accurate to call it \emph{the superior Lebesgue number}, rather than the largest one, but let us just call it \emph{the Lebesgue number}.

In any case, we are going to discuss the notion of \emph{the Lebesgue number} for a cover $\cU$ of a (non-empty) metric space $X$, where the cover $\cU$ is not necessarily open, the space $X$ is not necessarily compact, and the number itself, marked by $L(\cU)$, can take the values from $\left[0, \infty\right)\cup\{\infty\}$. Moreover, we shall introduce a \emph{relative} version of the Lebesgue number, which works for a (non-empty) subset $A$ of $X$ and takes into account whether $A\subseteq \cup\cU$ or $A=\cup\cU$ (see Definition \ref{def: Lebesgue relative} below). Then we intend to show a result (Lemma \ref{adjusted-lemma12.2.3}, which is a correction of \cite[Lemma 3.4]{BuyaloLebedeva}, and also \cite[Lemma 12.2.3]{BuyaloSchroeder}) that illustrates the importance of having this relative version of the Lebesgue number for subsets.

\section{Definition(s) of the Lebesgue number}

As a motivation for introducing our general definition, let us recall that one way of proving Lemma \ref{Leb-Version1} (see, for example, the proof of Theorem 5.26 in \cite{Kelley}),
goes as follows. Take any open cover $\cU$ of a compact metric space $X$ (such that $\cU$ does not contain $X$ itself), reduce it to a finite subcover $\{ U_1, \ldots , U_n \}$, and then define a function $f:X\to \left[0,\infty\right)$ by $\displaystyle f(x):=\max \{ \dist(x, X\setminus U_i)\ | \ i=1, \ldots , n\}$, which is continuous on $X$. Then $f(X)$ is a compact subset of positive real numbers, so it has a positive minimum, which we can take as a Lebesgue number $\lambda$ we are seeking (i.e., every open ball $B(x, \lambda)$ is contained in some element of $\cU$).

Generalizing this approach for a metric space $X$ which need not be compact, and a cover $\cU$ for $X$ which need not be open, we can replace the function $f$ written above by the function $L(\cU, \cdot): X \to \left[0,\infty\right) \cup \{\infty\}$ given by $\displaystyle L(\cU, x):= \sup_{ U\in\cU} \dist(x, X\setminus U)$, which is $\geq 0$ (it can be $=0$, depending on the choice of $\cU$), and then we can define the Lebesgue number of $\cU$ as the infimum of all $L(\cU, x)$ taken over all $x\in X$. That is, we have the following definition:

\begin{definition}[\cite{BuyaloSchroeder}, \cite{BuyaloLebedeva}]\label{def: Lebesgue BS}
Let $\cU$ be a cover of a non-empty metric space $X$, i.e., $X=\cup\cU$. Then we define
\begin{eqnarray}
L(\cU):=\inf_{x\in X}L(\cU,x), \ \text{ where } L(\cU,x):=\sup_{ U\in\cU} \dist(x, X\setminus U).
\end{eqnarray}
This number $L(\cU)$ is called \emph{the Lebesgue number of $\cU$}.
\end{definition}

In \cite{BuyaloSchroeder}, this definition is initially given for \emph{open} covers $\cU$ of $X$, but it is used on any covers of metric spaces, especially in the context of asymptotic dimension ($\asdim$). Let us state one of the equivalent definitions for $\asdim$ here.

\begin{definition}(\cite{BuyaloSchroeder}, \cite{BellDran})\label{def: asdim}
 For a metric space $X\neq\emptyset$, $\asdim X$ is the minimal $n\in \N\cup\{0\}$ such that for every $d>0$, there is a uniformly bounded cover $\cU$ of $X$ with multiplicity $\leq n+1$ and the Lebesgue number $L(\cU)>d$.
 \end{definition}
Since the relevant covers $\cU$ from Definition \ref{def: asdim} will have large $L(\cU)$, it will not matter whether these covers are open or not. Also, the requirement for covers to be uniformly bounded (i.e., the diameters of all sets in the cover to be bounded above by the same positive number) will mean that if $X$ is unbounded, the Lebesgue number will not be $\infty$.
In general,  if we want the Lebesgue number of a cover $\cU$ of $X$ to be finite, we need to assume that no member of $\cU$ equals $X$, since otherwise $\dist (x,X\setminus X)= \dist(x, \emptyset)=\infty$, so $L(\cU)=\infty$.

On the other hand, Definition \ref{def: Lebesgue BS} is allowing $L(\cU)$ to be $0$, however undesirable this may be, as this can happen for a cover $\cU$ of a non-compact metric space $X$, whether $\cU$ is open or not. Clearly, if we take $\cU=\{\left[n, n+1\right] \ | \ n \in \Z\}$ as a cover of $\R$, then $L(\cU)=0$. Or:

\begin{example}
If $X=\left(0,1\right)$ with the restriction of the standard Euclidean metric, and $\cU=\{ \left(\frac{1}{n}, 1\right) \ | \ n\in\N_{\geq 2} \}$, then for any $x\in \left(0,1\right)$, we have $L(\cU, x)=\sup \{ x -\frac{1}{n} \ | \ n\in\N_{\geq 2} \text{ and } \frac{1}{n}<x \}=x$, so $L(\cU)= \inf \{x \ | \ x\in (0,1) \}=0$.
\end{example}

\begin{remark}
There is a notion of a \emph{Lebesgue space}, which is a metric space having the property that every open cover of it has a (positive) Lebesgue number, taking the definition of a Lebesgue number from either Lemma \ref{Leb-Version1} or Lemma \ref{Leb-Version2} (see, for example, \cite{GiustoMarcone}). But we are not discussing Lebesgue spaces in this paper.
\end{remark}

\medskip

Considering how we arrived at Definiton \ref{def: Lebesgue BS} above, it is not surprising that the following is true:

\begin{lemma}\label{BS-def-L implies radius def}
Let $X$ be a metric space, and let $\cU$ be a cover of $X$ with the Lebesgue number $L(\cU)>0$. Then for each $r \in \left( 0, L(\cU)\right)$  and any $x\in X$, the open ball $B(x,r)$ has to be contained in some element of $\cU$, that is, the family of open balls $\{ B(x,r)\ | \ x\in X\}$ is a refinement of $\cU$.
\end{lemma}

\begin{proof}
Since $L(\cU)=\inf_{x\in X} L(\cU,x)$, from $r \in \left( 0, L(\cU)\right)$ it follows that $0< r< L(\cU, x)$, for all $x\in X$. Let us fix a random $x\in X$. From $r<L(\cU,x)=\sup_{U\in\cU} \dist (x, X\setminus U)$ it follows that there is a $U_{r} \in \cU$ such that $r < \dist (x, X\setminus U_{r})= \inf_{z\in X\setminus U_{r}} d(x,z)$. If $D:=\dist (x, X\setminus U_{r})$, then for each $z\in X\setminus U_{r}$ we have that $d(x,z)\geq D$. Note that $z\in X\setminus U_{r}$ implying $d(x,z)\geq D$ is equivalent to saying that, whenever $y\in X$ is such that $d(x,y)<D$, it follows that $y\in U_{r}$, i.e., $B(x,D)\subseteq U_{r}$. Therefore $B(x,r)\subset B(x, D)\subseteq U_{r}$.
\end{proof}

Let us point out that the radius $r$ from Lemma \ref{BS-def-L implies radius def} cannot be equal to $L(\cU)$ in general, as is shown in the next example.

\begin{example}\label{ex-balls-1}
For $X=\R^2$ with the standard Euclidean metric, and its cover by open balls $\cU=\{ B(x, 1 - \frac{1}{n}) \ | \ x\in \R^2, n \in \N_{\geq 2} \}$, we have that for all $x\in X$, $L(\cU, x)=\sup \{1-\frac{1}{n} \ | \ n\in\N_{\geq 2} \}=1$, so $L(\cU)=1$, but no open ball with radius $1$ is contained in any element of $\cU$.
\end{example}

Also note that $B(x,0)=\{y \in X \ | \ d(x,y)<0 \}=\emptyset$, so we could allow $r=0$ in Lemma \ref{BS-def-L implies radius def}, though it would not be informative.

\medskip

Regarding Lemma \ref{BS-def-L implies radius def}, we can also consider whether, for a metric space $X$ and its cover $\cU$, the following number is equal to our $L(\cU)$ from Definition~\ref{def: Lebesgue BS}:
\begin{definition}\label{def: L-Rad}
Let $\cU$ be a cover of a non-empty metric space $X$. Then we define
\begin{eqnarray}\label{L radius}
L_{Rad}(\cU) &:=& \inf_{x\in X}L_{Rad}(\cU,x), \ \text{ where} \nonumber \\
 L_{Rad}(\cU,x) &:=&\sup \ \{ r>0 \ | \ B(x,r)\subseteq U, \text{ for some } U\in \cU \}.
\end{eqnarray}
\end{definition}
First note that the definition of $L_{Rad}(\cU)$ only makes sense for covers $\cU$ such that, for every $x\in X$, there is some $r_x>0$ such that $B(x, r_x)$ is contained in some $U\in \cU$. So we could require any cover $\cU$ we work with to have this property, or simply ask for a cover $\cU$ in Definition \ref{def: L-Rad} to be open.
Let us now state the following lemma, comparing $L(\cU)$ and $L_{Rad}(\cU)$, for open covers:

\begin{lemma}\label{lemma: L vs L-Rad}
Let $X$ be a metric space and let $\cU$ be an open cover of $X$. Then $L(\cU)=L_{Rad}(\cU)$.
\end{lemma}
\begin{proof}
First note that, since $\cU$ is an open cover for $X$, then for each $x\in X$ we have both $L(\cU,x)>0$ and $L_{Rad}(\cU,x)>0$. This is because for $x\in X$ there is an open set $U\in \cU$ containing it, so there is an $r>0$ such that $B(x,r)\subseteq U$, which is equivalent to $\dist(x, X\setminus U)\geq r$.
Let us fix a random $x\in X$.

If we take any $r\in \left(0, L(\cU,x)\right)$, using the same argument as in the proof of Lemma \ref{BS-def-L implies radius def} we get that $B(x,r)\subseteq U_r$, for some $U_r\in \cU$, which means that all elements of $\left(0, L(\cU,x)\right)$ belong to the set $\cR_x:=\{ r>0 \ | \ B(x,r)\subseteq U, \text{ for some } U\in \cU \}$, i.e., $L(\cU,x)=\sup \left(0, L(\cU,x)\right)\leq \sup \cR_x = L_{Rad}(\cU,x)$.

On the other hand, if we take any $R\in \cR_x$, then there exists $U_R\in \cU$ such that $B(x, R)\subseteq U_{R}$, which is equivalent to saying that $\dist (x, X\setminus U_{R})\geq R$, so $L(\cU,x) =\sup_{U\in \cU} \dist (x, X\setminus U) \geq R$. Thus
$L(\cU,x) \geq \sup \cR_x = L_{Rad}(\cU,x)$.

Therefore $L(\cU,x)=L_{Rad}(\cU,x)$, for every $x\in X$, and the claim follows.
\end{proof}

Since Definition~\ref{def: Lebesgue BS} gives us $L(\cU)$ which is not as sensitive to the choice of covers as $L_{Rad}(\cU)$ is, we shall keep our focus on $L(\cU)$.

\bigskip

Another approach to introducing the Lebesgue number of a cover for a metric space would be to give a definition suggested by Lemma \ref{Leb-Version2}, as it was done by G.~Bell and A.~Dranishnikov in \cite{BellDran}. We will call their version $L_{Diam}(\cU)$, and define it here (with a slight adjustment in the wording, with respect to the originaly stated definition):
\begin{definition}\label{def: L-Diam} Let $\cU$ be a cover of a non-empty metric space $X$. Then we define
\begin{eqnarray}
L_{Diam}(\cU):= \sup \{D\geq 0\ | \ \forall A\subseteq X \text{\small{ with }} \diam A \leq D, \ \exists U\in \cU \text{ s.t. } A\subseteq U\}.
\end{eqnarray}
\end{definition}

For the correspondence between $L_{Diam}(\cU)$ and $L(\cU)$, we can prove the following:
\begin{lemma}\label{lemma: L vs L-Diam}
Let $X$ be a metric space and let $\cU$ be a cover of $X$. Then
\[ L(\cU)\leq L_{Diam}(\cU)\leq 2\cdot L(\cU).
\]
\end{lemma}
\begin{proof}
Let $\cU$ be a cover of $X$ ($\cU$ need not be open).
To begin with, suppose that $L(\cU)>0$. Then by Lemma \ref{BS-def-L implies radius def}, for any $\rho \in \left(0, L(\cU)\right)$ and any $x\in X$, the open ball $B(x,\rho)$ is contained in some $U_{\rho,x}\in \cU$. Let us fix an $r\in \left(0, L(\cU)\right)$. If we take any $A\subseteq X$ with $\diam A\leq r$,  then for any $x\in A$ we have $d(x,y)\leq r$, $\forall y\in A$, i.e., $A\subseteq \oB(x,r)$ (where $\oB(x,r)$ is the closed ball centered at $x$ with radius $r$). But note that, for our $r< L(\cU)$ there exists an $\varepsilon>0$ such that $r+\varepsilon < L(\cU)$, so we have $A\subseteq \oB(x,r)\subseteq B(x, r+\varepsilon)\subseteq U_{r+\varepsilon,x}$, for some $U_{r+\varepsilon,x}\in \cU$. Therefore $r\in \cD$, where
\[\cD :=\{D\geq 0\ | \ \forall A\subseteq X \text{\small{ with }} \diam A \leq D,  \ \exists U\in \cU \text{ s.t. } A\subseteq U\}.
\]
Thus we have $\left(0, L(\cU)\right) \subseteq \left( 0, \sup \cD\right)$, that is, $L(\cU)\leq L_{Diam}(\cU)$.

On the other hand, suppose that $L_{Diam}(\cU)=\sup \cD >0$, and take any $D\in \left( 0, L_{Diam}(\cU)\right)$. For a random $x\in X$, from $\diam \oB(x, \frac{D}{2})\leq D$ we get that there exists a $U_{D,x} \in \cU$ such that $\oB (x, \frac{D}{2})\subseteq U_{D,x}$, so $X\setminus \oB (x, \frac{D}{2}) \supseteq X\setminus U_{D,x}$, which implies $\dist (x, X\setminus U_{D,x})\geq \dist (x, X\setminus \oB (x, \frac{D}{2})) \geq\frac{D}{2}$. Therefore $L(\cU,x)=\sup_{U\in \cU} \dist(x, X\setminus U)\geq\frac{D}{2}$, for all $x\in X$, which means $L(\cU)=\inf_{x\in X} L(\cU,x)\geq \frac{D}{2}$, and the last inequality is true for all $D\in \left( 0, L_{Diam}(\cU)\right)$, giving us $L(\cU)\geq \frac{L_{Diam}(\cU)}{2}$.

Finally, if one of $L(\cU)$ or $L_{Diam}(\cU)$ is equal to $0$, so must the other be, since assuming the contrary would give us a contradiction with one of the inequalities above.
\end{proof}

\begin{remark}
From Lemma \ref{lemma: L vs L-Rad} and Lemma \ref{lemma: L vs L-Diam} it follows that, for an open cover $\cU$ of $X$, we have 
\[ L_{Rad}(\cU)\leq L_{Diam}(\cU)\leq 2\cdot L_{Rad}(\cU).
\]
\end{remark}

Let us also note here that Lemma \ref{lemma: L vs L-Diam} helps us understand that, as far as Definition \ref{def: asdim} for asymptotic dimension is concerned, using $L(\cU)$ (as in \cite{BuyaloSchroeder}) or $L_{Diam}(\cU)$ (as in \cite{BellDran}) as definition of the Lebesgue number of $\cU$ does not make a difference for $\asdim X$.

\section{The relative Lebesgue number for a covering family of a subset}

Before we introduce the Lebesgue number with respect to subsets (i.e., subspaces) of a metric space, we need some additional terminology, which was also used in \cite{CHT}.

\begin{definition}
If $X\neq\emptyset$ is a set, and $A$ is any non-empty subset of $X$, we say that a family $\cU$ of subsets of $X$ is a \emph{covering family of $A$ in $X$} if $A \subseteq \cup \cU$.
\end{definition}

Note that $\cU$ being a covering family of $A$ in $X$ is equivalent to $\cU|_A$ being a cover of $A$, where 
$\cU|_A=\{U\cap A\ | \ U\in\cU\}$, i.e., $A=\cup \cU|_A$. Clearly, any cover of $X$ is also a covering family of any subset of $X$ (including $X$ itself). Also note that if we require a covering family $\cU$ of $A$ to be an open covering family, this will mean that all $U\in \cU$ are open in the ambient space $X$.
 The word \emph{cover} is often used in literature for both $A=\cup\cU$ and $A\subseteq \cup\cU$, but we find making the distinction between a cover and a covering family helpful.
 We are now ready to introduce the \emph{relative} version of the Lebesgue number for covering families of subsets of $X$.
 
\begin{definition}\label{def: Lebesgue relative}
 If $A\neq\emptyset$ is a subset (i.e., a subspace) of a metric space $X$, then \emph{the Lebesgue number of a covering family $\cU$ of $A$} in the ambient space $X$ is defined as 
\begin{eqnarray}
L_X(\cU, A):=\inf_{x\in A}L_X(\cU,x), \ \text{ where } L_X(\cU,x):=\sup_{ U\in\cU} \dist(x, X\setminus U).
\end{eqnarray}
\end{definition}

Note that $L_X(\cU,x)$ from Definition \ref{def: Lebesgue relative} is the same as $L(\cU,x)$ from Definition \ref{def: Lebesgue BS} (though our $\cU$ here need not cover the entire $X$), but we wanted to emphasize the ambient space by adding the index. Also note that
 for $A=X$ we have
\[ L_X(\cU, X) = \inf_{x\in X}L_X(\cU,x) = L(\cU), \]
that is, in the case $A=X$ we recover the Definition \ref{def: Lebesgue BS} of the Lebesgue number of a cover $\cU$ of $X$.
Here is a lemma comparing the notion of $L(\cU)$ with various appropriate relative versions.

\begin{lemma}\label{lemma: Lebesgue-restriction-of-family}\label{Lebesgue-subspace}
When $\cU$ is a cover of a metric space $X$ and $A\subseteq B\subseteq X$, then
\begin{eqnarray}\label{Lebesgue-restriction-of-family}
L(\cU) = L_X(\cU, X)\leq L_X (\cU,A) \leq L_B(\cU|_B, A)\leq L_A (\cU|_A, A). 
\end{eqnarray}
The middle and the rightmost inequalities are also true for a covering family $\cU$ of $A$ in $X$.
\end{lemma}
\begin{proof} $L(\cU) = L_X(\cU, X)\leq L_X (\cU,A)$ is immediate from the definitions. For the middle and the rightmost inequality (whether $\cU$ is a cover of $X$, or a covering family of $A$ in $X$), note that for $U \in \cU$ we have
$A\setminus U\subseteq B\setminus U \subseteq X\setminus U$, and therefore $\dist(x, A\setminus U)\geq \dist (x, B\setminus U)\geq \dist (x, X\setminus U)$ for all $x\in A$, $U\in\cU$. Passing to the $\sup_{U\in\cU}$ yields $L_A(\cU|_A,x)\geq L_B(\cU|_B,x)\geq L_X(\cU,x)$ for all $x\in A$, so the middle and the rightmost inequalities follow after applying $\inf_{x\in X}$.
\end{proof}

It is easy to show that in the situation of Lemma \ref{lemma: Lebesgue-restriction-of-family}, the inequalities can be strict. Let us see an example of this for the middle and the rightmost inequalities, for a covering family of a subset of $X$.

\begin{example}
Let $X=\R^3$ with Euclidean metric, $A=\left[ 0,1\right]^2\times \{0\}$, $B=\R^2\times \{0\}$, and take \\
$\cU=\left\{U_1=\left(-\frac{1}{4}, \frac{7}{8}\right)\times \left(-\frac{1}{4}, \frac{5}{4}\right) \times \left(-\frac{1}{8}, \frac{1}{8}\right),
 U_2=\left(\frac{1}{8}, \frac{5}{4}\right)\times \left(-\frac{1}{4}, \frac{5}{4}\right) \times \left(-\frac{1}{8}, \frac{1}{8}\right) \right\}$ to be an open covering family of $A$ in $X$.
Then $L_A(\cU|_A, A)=\inf_{a\in A} \sup_{U\in\cU} \dist (a, A\setminus U) =\frac{3}{8}$,
$L_B(\cU|_B, A)=\inf_{a\in A} \sup_{U\in\cU} \dist (a, B\setminus U) =\frac{1}{4}$, 
while $L_X(\cU, A)=\inf_{a\in A} \sup_{U\in\cU} \dist (a, X\setminus U) =\frac{1}{8}$.
\end{example}

\bigskip

Now let us recall the notion of \emph{mesh} for a family of subsets in a metric space.

\begin{definition} If $X$ is a metric space and $\cU$ is any family of subsets of $X$ (in particular, a cover of $X$), then we define the \emph{mesh} of $\cU$ by $\mesh\cU:=\sup\{\diam U \mid U\in\cU\}$. 
\end{definition}
Clearly, if $A\subseteq X$ and $\cU$ is a cover of $X$, we have 
\begin{equation}\label{eq: mesh restr}
\mesh \cU\geq \mesh (\cU|_A).
\end{equation}
Also note that if $\cU'$ is a subcover of the cover $\cU$ of $X$, then it is easy to see that
\[
L(\cU') \leq L(\cU) \quad \text{and}\quad \mesh\cU' \leq \mesh \cU.
\]

Concerning the relation between the Lebesgue number and the mesh for the same cover, we observe that while the Lebesgue number (from Def.~\ref{def: Lebesgue BS}) is well-behaved with respect to subspaces in the sense of Lemma \ref{lemma: Lebesgue-restriction-of-family}, it is not bounded above by mesh, as is shown by the following:

\begin{example} \label{example: discrete} Let $X$ be a finite set with discrete metric, i.e., the metric is given by $d(x,y) = 1$ for $x \neq y$. Let $\cU$ be the open cover of $X$ by singletons. Then $\mesh \cU = 0$ and $L(\cU) = 1$. Thus in general the Lebesgue number of a cover is not bounded above by the mesh of that cover, and this issue cannot be resolved by passing to a subcover (since the singleton cover does not contain any proper subcovers). 
\end{example}

On the other hand, we leave it to the reader to check that, if $\cU$ is an open cover for a \emph{connected} metric space (such that no element of $\cU$ is equal to entire $X$), then for every $x\in X$, we have  $L(\cU, x)\leq\mesh \cU$, so here the Lebesgue number $L(\cU)$ is bounded above by $\mesh \cU$.
But this is not inherited by all subspaces, since we can certainly choose a subspace of a connected space $X$ that looks like the space described in Example \ref{example: discrete}.
 That is, even if we start from a cover $\cU$ of $X$ such that $L(\cU) \leq \mesh\cU$,  it may still happen for a subspace $B \subsetneq X$ that
$ L(\cU|_B) > {\rm mesh}(\cU|_B)$. So even if a cover $\cU$ of a space $X$ has $L(\cU)$ bounded by $\mesh \cU$, we inevitably run into problems when we want to control the Lebesgue number of restrictions of a given cover to subspaces, in terms of the corresponding meshes.

\medskip

Thus we should mention here that there is an alternative definition of the Lebesgue number which has a built-in control by mesh (see, for example, \cite{Buyalo}, with notation adjusted to ours):
\begin{definition} \label{def: Lebesgue w mesh} Let $X$ be a metric space. For an (open) cover $\cU$ of $X$ we define 
\begin{eqnarray}
\widetilde L(\cU,x):=\min \{ L(\cU, x), \mesh (\cU,x) \}, \text{ where } \mesh(\cU, x) :=\sup \{\diam  U\ | \ x\in U\in \cU\},
\end{eqnarray}
and $L(\cU,x)$ is the same as in Definition \ref{def: Lebesgue BS}. Then we define \emph{the Lebesgue number of $\cU$ of the second kind} by
\[
\widetilde L(\cU):=\inf_{x\in X}\widetilde L(\cU,x).
\]
\end{definition}

It is clear that $\mesh\cU =\sup_{x\in X} \mesh (\cU,x)$, and that from $\widetilde L(\cU,x)\leq \mesh (\cU,x)$ we get $\widetilde L(\cU)\leq \mesh \cU$.

As in Definition \ref{def: Lebesgue relative}, for any $A\subseteq X$ we could also introduce the relative version $\widetilde{L}_X(\cU, A)=\inf_{x\in A}\widetilde L(\cU,x)$ of Definition 
\ref{def: Lebesgue w mesh}, but because of the built-in mesh control, the Lebesgue number of the second kind fails to satisfy the inequality~\eqref{Lebesgue-restriction-of-family} from Lemma~\ref {lemma: Lebesgue-restriction-of-family} (the inequality\eqref{eq: mesh restr} gets in the way). Since the inequality \eqref{Lebesgue-restriction-of-family} is
often useful in proofs of theorems concerning a\-sym\-ptotic dimension, as well as other types of dimension,
 the Lebesgue number of the second kind will not be discussed further in this paper.

\section{An application of the relative Lebesgue number}

In this section, we would like to show the advantage of introducing the notion of the relative Lebesgue number, which comes with the precise notation for the subset, its covering family and the ambient space. We will show that, using the precise notation, 
we will be able to give a more precise, corrected formulation and the proof of 
Lemma 3.4 from \cite{BuyaloLebedeva}, which is also, in a slightly different form, Lemma 12.2.3 in \cite{BuyaloSchroeder}. The original statement of this lemma reads as follows (considering $Z$ and $Z'$ are metric spaces):

\begin{lemma}
\cite[Lemma 3.4]{BuyaloLebedeva}\label{lemma: BL wrong}
Let $h:Z\to Z'$ be a $\lambda$--quasi-homothetic map with coefficient $R$. Let $V\subseteq Z$, let $\widetilde\cU$ be an open covering\footnote{not mentioned in which ambient space} of $h(V)$, and let $\cU=h^{-1}(\widetilde\cU)$. Then: 
\begin{enumerate}
\item[(i)] $\frac{1}{\lambda}R\cdot \mesh\cU\leq \mesh\widetilde\cU \leq \lambda R\cdot \mesh\cU$;
\item[(ii)] $\frac{1}{\lambda} R \cdot L(\cU)\leq L(\widetilde\cU)\leq \lambda R\cdot L(\cU)$, where $L(\cU)$ is the Lebesgue number of $\cU$ as a covering of $V$.
\end{enumerate}
\end{lemma}

In \cite {BuyaloSchroeder}, the same lemma appears as Lemma 12.2.3, with part (ii) changed as follows:

\emph{(ii') $L(\widetilde\cU) \leq \lambda R \cdot L(\cU)$, where $L(\cU)$ is the Lebesgue number of $\cU$ as a covering of $V$.}

\medskip

This lemma is not proven in either of the sources, as it is considered to be a direct consequence of definition of a quasi-homothetic map.
So let us first recall the following definition:

\begin{definition}\label{def: q-homothetic}
For some $\lambda\geq 1$ and $R>0$, a map $h:A\to B$ between metric spaces is called \emph{$\lambda$--quasi-homothetic with coefficient $R$} provided that
\begin{equation}\label{lambda-quasi-homothetic}
\displaystyle \frac{1}{\lambda} R \cdot d_A(a_1,a_2) \ \leq \ d_B(h(a_1),h(a_2)) \ \leq \ \lambda R\cdot d_A(a_1,a_2), \quad \forall a_1,a_2 \in A.
\end{equation}
\end{definition}
In particular, any such map is continuous and injective. 

\medskip

Now to go back to the statement of Lemma \ref{lemma: BL wrong},
the problem with it is the phrase ``$\widetilde\cU$ is an open covering of $h(V)$'', which is open to interpretation regarding the ambient space for $h(V)$ (which is not mentioned in \cite{BuyaloLebedeva} nor in \cite{BuyaloSchroeder}). This $\widetilde\cU$ could be an open cover of $h(V)$ in $h(V)$ (i.e. $h(V)=\cup \widetilde\cU$), or more generally, $\widetilde\cU$ could be an open covering family of $h(V)$ in $h(Z)$, or even more generally an open covering family of $h(V)$ in $Z'$.

When this lemma is used to prove Theorem 1.1 of \cite{BuyaloLebedeva} (which is also Theorem 12.2.1 of \cite{BuyaloSchroeder}), it turns out it is used for $\widetilde\cU$ being an open covering family of $h(V)$ in $Z'$. Therefore, after adjusting for our notation, the precise statement of Lemma \ref{lemma: BL wrong}, as used in \cite{BuyaloLebedeva}, is:

\begin{lemma}[Lemma \ref{lemma: BL wrong}, stated precisely]\label{lemma: BL wrong precise}
Let $h:Z\to Z'$ be a $\lambda$--quasi-homothetic map with coefficient $R$. Let $V\subseteq Z$, let $\widetilde\cU$ be an open covering family of $h(V)$ in $Z'$, and let $\cU=h^{-1}(\widetilde\cU)$. Then:
\begin{enumerate}
\item[(i)] $\frac{1}{\lambda}R\cdot \mesh\cU\leq \mesh\widetilde\cU \leq \lambda R\cdot \mesh\cU$, \ and
\item[(ii)] $\frac{1}{\lambda}R\cdot L_Z(\cU,V)\leq L_{Z'}(\widetilde\cU, h(V)) \leq \lambda R\cdot L_Z(\cU, V)$.
\end{enumerate}
\end{lemma}

But note that the right side inequality of (i) and the left side inequality  of (ii) of Lemma \ref{lemma: BL wrong precise} are incorrect, 
which is easily shown by examples, like the following one.

\begin{example}\label{example: R2-R3} Take $\R^2$ and $\R^3$ with the Euclidean metric on both, and
let $h:\R^2 \to \R^3$ be the embedding $h(x_1,x_2)=(x_1,x_2,0)$, which is a $1$-quasi-homothetic map with coefficient $1$. Let $V=[0,1]^2\subset \R^2$.  
Let $\widetilde\cU=
\{\widetilde U_1=\left(-\frac{1}{4}, \frac{3}{4}\right)\times \left(-\frac{1}{4}, \frac{5}{4}\right) \times \left(-\frac{1}{8}, \frac{1}{8}\right),
\widetilde U_2=\left(\frac{1}{4}, \frac{5}{4}\right)\times \left(-\frac{1}{4}, \frac{5}{4}\right) \times \left(-\frac{1}{8}, \frac{1}{8}\right) \}$ 
be an open covering family of $h(V)=[0,1]^2\times \{0\}$ in $\R^3$.
 Then $\cU =h^{-1}(\widetilde\cU)=\{U_1=\left(-\frac{1}{4}, \frac{3}{4}\right)\times \left(-\frac{1}{4}, \frac{5}{4}\right),
U_2=\left(\frac{1}{4}, \frac{5}{4}\right)\times \left(-\frac{1}{4}, \frac{5}{4}\right) \}$, so 
 $\mesh \cU=\sqrt{3+\frac{1}{4}}$, $\mesh \widetilde\cU=\sqrt{3+\frac{1}{4}+\frac{1}{16}}$, and $\mesh\widetilde\cU \nleq 1\cdot1\cdot \mesh\cU$.\\
Also,
$L_Z(\cU, V)=L_{\R^2}(\cU, V)=\inf_{(x_1,x_2)\in V} \sup_{U\in \cU} \dist \left((x_1,x_2), \R^2 \setminus U\right)=\frac{1}{4}$,
and \\
$L_{Z'}(\widetilde\cU, h(V))=L_{\R^3}(\widetilde\cU, h(V))=\inf_{(x_1,x_2, 0)\in h(V)} \sup_{\widetilde U\in \widetilde\cU} \dist \left((x_1,x_2,0), \R^3 \setminus \widetilde U\right)=\frac{1}{8}$.\\
 Therefore $\frac{1}{1}\cdot 1\cdot L_Z(\cU, V)\nleq L_{Z'} (\widetilde\cU, h(V))$.
\end{example}

We should note here that the version of Lemma \ref{lemma: BL wrong} published in \cite{BuyaloSchroeder} as Lemma 12.2.3 omits the wrong part of the inequality in part (ii), while the issue with part (i) remains. In any case, the story of Lemma \ref{lemma: BL wrong} shown so far illustrates the importance of introducing detailed notation for the relative Lebesgue number (which includes a covering family, a set to which it is applied and the ambient space for it all), as well as being careful with statements involving them. Now we are ready to state
the corrected, precise version of Lemma \ref{lemma: BL wrong}, and prove it, noting that in it, $\widetilde\cU$ is an open covering family of $h(V)$ in the ambient space $h(Z)$. The proof of this lemma, once it is stated correctly, is indeed a consequence of Definition \ref{def: q-homothetic} for a quasi-homothetic map.

\begin{lemma}\label{adjusted-lemma12.2.3}
Let $h:Z\to Z'$ be a $\lambda$--quasi-homothetic map with coefficient $R$. Let $V\subseteq Z$, let $\widetilde\cU$ be an open covering family of $h(V)$ in $h(Z)$, and let $\cU=h^{-1}(\widetilde\cU)=\{h^{-1}(\widetilde U)\ | \ \widetilde U\in \widetilde\cU\}$. Then:
\begin{enumerate}
\item[(i)] $\frac{1}{\lambda}R\cdot \mesh\cU\leq \mesh\widetilde\cU \leq \lambda R\cdot \mesh\cU$, \ and
\item[(ii)] $\frac{1}{\lambda}R\cdot L_Z(\cU,V)\leq L_{h(Z)}(\widetilde\cU, h(V)) \leq \lambda R\cdot L_Z(\cU, V)$.
\end{enumerate}
\end{lemma}

\begin{proof}
First note that $h:Z\to Z'$ is injective, so $h$ taken as a function $:Z\to h(Z)$ is bijective.
Also, $h$ is continuous, so $\cU=h^{-1}(\widetilde\cU)$ is an open covering family of $V$ in $Z$. Therefore there is a bijective correspondence between elements of $\widetilde \cU (\subseteq h(Z))$ and $\cU (\subseteq Z)$: each $\widetilde U$ from $\widetilde\cU$ corresponds to $U=h^{-1}(\widetilde U)$ from $\cU$, and $h(U)=\widetilde U$, that is, $h(\cU)=\widetilde\cU$. \\
For (i): 
\begin{eqnarray*}\displaystyle \mesh \widetilde\cU & = & \sup_{\widetilde U \in \widetilde\cU} \diam \widetilde U=\sup_{\widetilde U \in \widetilde\cU}\  \sup_{b_1, b_2 \in \widetilde U}d_{Z'}(b_1,b_2)
 = \sup_{h( U) \in h(\cU)}\  \sup_{b_1, b_2 \in h( U)}d_{Z'}(b_1,b_2)\\
& = & \sup_{U \in \cU}\  \sup_{a_1, a_2 \in  U}d_{Z'}(h(a_1),h(a_2)) \overset{\eqref{lambda-quasi-homothetic}}\leq \sup_{U \in \cU}\  \sup_{a_1, a_2 \in  U}\lambda R\cdot d_Z(a_1,a_2) \\
& = & \lambda R\cdot \sup_{U\in \cU}\diam U = \lambda R\cdot \mesh\cU.
\end{eqnarray*}
The other inequality for $\mesh$ is proven analogously.\\
For (ii):  
\begin{eqnarray*}\displaystyle
L_{h(Z)}(\widetilde  \cU, h(V)) & = & \inf_{b\in h(V)} L_{h(Z)}(\widetilde\cU, b)=\inf_{b\in h(V)}\  \sup_{\widetilde U\in \widetilde\cU}\ \dist_{Z'}(b,h(Z)\setminus\widetilde U)\\
& = & \inf_{h(a)\in h(V)}\  \sup_{h( U)\in h(\cU)}\ \dist_{Z'}(h(a),h(Z)\setminus h(U))\\
& = & \inf_{a\in V}\  \sup_{U\in \cU}\ \dist_{Z'}(h(a),h(Z\setminus U))=\inf_{a\in V}\  \sup_{U\in \cU} \ \inf_{x\in Z\setminus U} d_{Z'}(h(a), h(x))\\
& \overset{\eqref{lambda-quasi-homothetic}}\geq & \inf_{a\in V}\  \sup_{U\in \cU} \ \inf_{x\in Z\setminus U} \frac{1}{\lambda} R\cdot d_Z(a,x)=\frac{1}{\lambda} R \cdot \inf_{a\in V}\  \sup_{U\in \cU} \dist_Z (a, Z\setminus U)\\
& = & \frac{1}{\lambda} R \cdot \inf_{a\in V} L_Z(\cU, a)= \frac{1}{\lambda} R\cdot L_Z(\cU, V).
\end{eqnarray*}
The other inequality for the Lebesgue number is proven analogously.
\end{proof}

\begin{remark}
Let us note here that the issues with Lemma \ref{lemma: BL wrong} (i.e., \cite[Lemma 3.4]{BuyaloLebedeva}, and also \cite[Lemma 12.2.3]{BuyaloSchroeder})
cause a gap in the original proof of  Theorem 1.1 of \cite{BuyaloLebedeva} (also Theorem 12.2.1 of \cite{BuyaloSchroeder}). But the theorem in question 
can be proven using Lemma \ref{adjusted-lemma12.2.3}, i.e., the corrected version of Lemma \ref{lemma: BL wrong}, together with some other corrections, shown to the author of this paper by Nina Lebedeva in private correspondence.
\end{remark}

\end{document}